\documentclass{LMCS}

\input{amssym.def}
\input{amssym.tex}
\usepackage{enumerate}
\usepackage{hyperref}

\newcommand{\R}{{\Bbb R}}
\newcommand{\N}{{\Bbb N}}

\newcommand{\Q}{{\Bbb Q}}

\newcommand{\rinf}{\rightarrow \infty}

\newcommand{\dminus}{\mbox{$\;^\cdot\!\!\!-$}}

\def\doi{7 (2:11) 2011}
\lmcsheading%
{\doi}
{1--20}
{}
{}
{Apr.~26, 2010}
{May\phantom.~17, 2011}
{}

\begin{document}
\title{Banach spaces as data types}
\author[D.~Normann]{Dag Normann}
\address{Department of Mathematics, The University 
of Oslo, P.O. Box 1053, Blindern N-0316 Oslo, Norway}  
\email{dnormann@math.uio.no}
\keywords{Banach spaces, computability, internal algorithms}
\subjclass{F.1.1, F.4.1, F.3.2}

\begin{abstract}We introduce the operators {\em modified limit} and
{\em accumulation} on a Banach space, and we use this to define what
we mean by being internally computable over the space. We prove that
any externally computable function from a computable metric space to a
computable Banach space is internally computable.

We motivate the need for internal concepts of computability by
observing that the complexity of the set of finite sets of closed
balls with a nonempty intersection is not uniformly hyperarithmetical,
and thus that approximating an externally computable function is
highly complex.
\end{abstract}

\maketitle

\section{Introduction}\label{Sec1}This note may be viewed as a sequel 
to what we once called ``the trilogy" Normann \cite{D1,D2,D3}. There
we were interested in typed structures of continuous functions over
complete separable metric spaces. We used the, for our purpose
equivalent, cartesian closed categories of ${\bf QCB}_0$ (\cite{BSS})
and Kuratowski Limit Spaces (\cite{Ku}).

In computational analysis it is customary to work with
TTE-representations, see, e.g., Weihrauch~\cite{We}, or alternatively with
domain representations. The main observation in \cite{D1} was that if
we restrict ourselves to the limit space representation and ignore
TTE-representations or domain representations, we obtain more
informative results with simpler proofs. In \cite{D2} we discussed the
distinction between internal and external approaches to computability
over a structured set.  The internally computable functions are
defined from elements, relations and functions present in the
structure, using acceptable operators that form new functions. The
problem will be to decide what the acceptable operators are. The
external concepts will be based on an already established definition
of computability on a set of representatives. In our view, internal
approaches are worthwhile pursuing. The advantage with an internal
approach is on the one hand that one does not have to translate
everything to the set of representatives, and on the other hand that
an internally defined object will always be well defined. When there
is a natural concept of external computability around, every
internally computable function should also be externally computable,
this is a soundness criterion. We refer to \cite{D2} for a further
discussion and for references.

In Section \ref{Sec2} we will prove that for computable Banach spaces in general it is very hard to decide when three closed balls  have a nonempty intersection. Our precise statement is that the problem of when three closed balls have a nonempty intersection is not uniformly hyperarithmetical. We will use a computable version of a theorem due to R. Kaufman \cite{Kauf}, combined with a trick due to Alfsen and Effros \cite{AE}, in order to establish this.  The conclusion we draw is that when $X$ is a computable metric space and $M$ is a computable Banach space, it is not hyperarithmetically decidable in general when a given finite approximation to a function $F:\N^{\N} \rightarrow \N^{\N}$ also approximates a representation of a computable function from $X$ to $M$. We see this as a support for introducing internal concepts of computability in analysis.

\begin{rem} The hyperarithmetical hierarchy, introduced 
by Kleene, is a transfinite extension of the arithmetical hierarchy,
see Rogers \cite{Rogers} or Sacks \cite{Sacks} for an
introduction. The hyperarithmetical subsets of $\N^\N$ will be exactly
the $\Delta^1_1$-sets, and they may be considered to be the Borel sets
where the construction can be coded by a computable, well founded
tree.
\end{rem}

When a separable Banach space $M$ is a closed subset of a metric space
$X$, then $M$ is actually a retract of $X$. In Section \ref{Sec3} we
will prove this, consider the proof as a case study and discuss on the
basis of the proof why we might consider the retract as internally
computable from the data at hand.  The internal principle we will use
can be seen as a combination of primitive recursion and forming
effective limits. 

 Retracts at base types induce retracts at
higher types. We use this to observe the effective version of the
embedding theorem from \cite{D3}: Let $U$ be the universal separable
metric space introduced by Urysohn in \cite{U1,U2}. If $\sigma(\vec
X)$ is a type in base type variables $\vec X$, and $\vec M$ is a
corresponding sequence of computable Banach spaces, then there is a
computable topological embedding of $\sigma(\vec M)$ into $\sigma(\vec
U)$. We will give a brief introduction to $U$ in Section \ref {Sec3}.

At the end of Section \ref{Sec3} we will discuss another scheme for
internal algorithms, inspired by the $\mu$-operator, and we will show
that every externally computable function from a computable metric
space to a computable Banach space will be internally computable in
this extended sense.

\section{When do three closed balls in a Banach space have a nonempty intersection? }\label{Sec2}
\subsection{Discussion}
In  this section we will consider Banach spaces that are  completions
of countable normed vector spaces over $\Q$, and the motivating
problem is:\medskip

\begin{center}{\em Given three closed balls with known radiuses and centers, how hard is it to decide if they have a point in common?}\end{center}\medskip

\noindent We will make the question precise, and show that the problem is of analytic, but not of Borel, complexity. An effective version will be that it is $\Sigma^1_1$, but not hyperarithmetical.
In the proof we have used an elaborated  version of a construction due to Alfsen and Effros \cite{AE}, see Lemma \ref{AE}. 

A computable metric space is given as the completion of a computable metric $d$ on an enumerated set $\{x_n\mid n \in \N\}$. An object in a computable metric space is represented by functions $\gamma:\N \rightarrow \N$ such that $$ d(x_{\gamma(n)},x_{\gamma(n+1)}) \leq 2^{-n}$$
for all $n \in \N$.  We say that the sequence represented by $\gamma$ is {\em fast converging}. We will use $\gamma$, $\xi$ etc. for {\em representing functions}. A function $f$ from one computable metric space $X$ to another space $Y$ is {\em externally computable} if there is a computable partial functional $ F:\N^{\N }\rightarrow \N^{\N}$ that will be total on the set of functions representing elements in $X$, and will map a function representing $x \in X$ to a function representing  $f(x) \in Y$.
A {\em computable Banach space} is primarily given as a computable metric space with computable $+$ and multiplication with scalars, respecting that the metric is given by a norm. 

These definitions are standard in the TTE-approach to computational analysis, and the concepts are extensional ones. 

Kaufman \cite{Kauf} showed that whenever $M$ is a nonreflexive, separable Banach space, then there is an equivalent norm on $M$ such that the set of bounded linear functionals attaining this norm is complete $\Sigma^1_1$ in the weak star topology. As stated, this is stronger than our Theorem \ref{BO}, and it can be used to prove that the collection of triples of closed balls having a nonempty intersection is complete analytical.
We give an alternative proof for two reasons\medskip
\begin{enumerate}[$-$]
\item We need an effective version of the theorem for our
applications and we find it simpler to give a direct proof for the
special case we need than to analyze the computational content of
Kaufman's proof, even for the special cases given in his
paper.\medskip

\item Our argument is quite different in flavor than that of Kaufman, and we think that the method of proof is of interest in itself. In a sense, our argument is more elementary, but then it leads to a definitely weaker result.
\end{enumerate}\medskip

\noindent In this paper, all Banach spaces are real, and all linear functionals are real. Following the tradition in set theory and theoretical computer science, we let the natural numbers $\N$ start with $0$.
\subsection{Closed Balls}
The problem under investigation is the following: Given a real Banach space $M$ via a norm on a dense, countable subspace $N$, where $N$ is a vector space over $\Q$, and given $n$ closed balls $B_1, \ldots,B_n$ with centers in $N$ and radiuses in $\Q$, how hard is it to tell if the intersection $$B_1 \cap \cdots \cap B_n$$ is empty or not.
If $n = 1$ or $n=2$, this is trivial. We will show that we obtain the full complexity for $n = 3$.
Clearly the statement is analytic, or $\Sigma^1_1$, since it is a simple statement of the existence of a point satisfying an arithmetical requirement. We will show that it is not Borel, by showing that every Borel subset of $\N^\N$ can be continuously reduced to this statement.

In  \cite{AE}, Alfsen and Effros gave an example of a Banach space $M$ with a subspace $M_0$  and three closed balls $B_1$, $B_2$ and $B_3$ with radius 1 and centers in $M_0$ such that these balls do not intersect in $M_0$ but in $M$. Our first lemma is based on that example.
\begin{lem}\label{AE}
Let $M$ be a Banach space with norm $|| \cdot ||_M$ and let $F:M \rightarrow \R$ be a linear functional with norm 1.
We let $M_F$ be $\R^2 \times M$ with the following norm:
$$||(x,y,z)|| = \max\{|x|,|y|, ||z||_M , |x + y + F(z)|\}.$$
Let $B_1$, $B_2$ and $B_3$ be the balls in $M_F$ with radius 1 and centers in $(2,0,0_M)$, $(0,2,0_M)$ and $(2,2,0_M)$ respectively. 
Then the following are equivalent:
\begin{enumerate}[\em(1)]
\item $B_1 \cap B_2 \cap B_3 \neq \emptyset$.
\item For some $z$ in the unit sphere of $M$ we have that $|F(z)| = 1$.
\end{enumerate}
\end{lem}
\begin{proof}
First, let $(x,y,z) \in B_1 \cap B_2 \cap B_3$. Then $x = y = 1$ and $||z||_M \leq 1$.
Considering $$||(2,2,0_M) - (1,1,z)|| = ||(1,1,-z)|| \leq 1$$
we see that  we also must have that $|2 - F(z)| \leq 1$, and the only possibility then is that $||z||_M = 1$ and $F(z) = 1$. This proves 1. $\Rightarrow$ 2.
Conversely, we see by direct calculation that if $x = y = ||z||_M = F(z) = 1$, then $(1,1,z)$ is in the intersection.
\end{proof}
\subsection{Linear functionals}
In the previous section, we did relate the intersection property to a property of linear functionals.
\begin{defi}Let $M$ be a Banach space and let $F:M \rightarrow \R$ be a linear functional with norm 1.
We say that $F$ {\em attains its norm} if for some $z$ on the unit
sphere of $M$, $F(z) = 1$.\end{defi}

\subsubsection{Coding quantifiers and truth values}
Now, let $M_n$ be a Banach space for each $n$, and  let $F_n$ be a linear functional on $M_n$ with norm 1. Let $|| \cdot ||_n$ be the norm on $M_n$.
We will construct two spaces $M_\exists$ and $M_\forall$ with corresponding norms and linear functionals $F_\exists$ and $F_\forall$, and the point is that 
\begin{enumerate}[$-$]
\item $F_\exists$ will attain its norm if and only if $F_n$ attains its norm for some $n$,
\item $F_\forall$ will attain its norm if and only if $F_n$ attains its norm for all $n$.
\end{enumerate}
\paragraph{\em The existential quantifier}
We let $M_\exists$ consist of all functions $$f \in \prod_{n \in \N} M_n$$ such that
$$\sum_{n \in \N} ||f(n)||_n < \infty.$$
We let 
$$||f||_\exists = \sum_{n \in \N} ||f(n)||_n$$
and we let 
$$F_\exists(f) = \sum_{n \in \N}F_n(f(n)).$$
If $x$ is in the unit sphere of $M_n$ such that $F_n(x) = 1$, we let $f(n) = x$ and $f(m) = 0_{M_m}$ for $m \neq n$.  Then $f$ is in the unit sphere of $M$ and $F_\exists(f) = 1$. Thus if $F_n$ attains its norm for at least one $n$, then $F_\exists$ will do so.

Conversely, if $f$ is in the unit sphere of $M_\exists$, and $F_\exists(f) = 1$, then we must have that $f(n)$ is nonzero for at least one $n$ and that $F_n(f(n)) = ||f(n)||_n$ for all $n$, since otherwise we will have a proper inequality $F_\exists(f) < ||f||_\exists$.
It is easy to see that the norm of $F_\exists$ must be 1.
\paragraph{\em The universal quantifier}
We let $M_\forall$ consist of all functions $$f \in \prod_{n \in \N}M_n$$ such that $$\sum_{n \in \N}2^{-n}||f(n)||_n < \infty.$$
Let $$||f||_\forall = \sum_{n \in \N}2^{-(n+1)}\max\{||f(n)||_n,||f(n+1)||_{n+1}\}$$
and we let $$F_\forall(f) = \sum_{n \in \N}2^{-(n+2)}(F_n(f(n)) + F_{n+1}(f(n+1)).$$
It is easy to see that the norm is well defined, that $F_\forall$ is well defined and that it has norm 1.

Now, assume that $x_n$ is in the unit sphere of $M_n$ and is such that $F_n(x_n) = 1$ for each $n \in \N$. Let $f(n) = x_n$. Then $||f||_\forall = 1$ and $F_\forall(f) = 1$, so $f$ will witness that $F_\forall$ will attain its norm.

For the converse, assume that for some $f$ in the unit sphere of $M_\forall$ we have $F_\forall(f) = 1$.
Then we in particular must have for all $n \in \N$ that
$$\max\{||f(n)||_n,||f(n+1)||_{n+1}\} = \frac{1}{2}(F_n(f(n))+F_{n+1}(f(n+1))),$$ 
since if this is not the case, we will have the proper inequality $$|F_\forall(f)| < ||f||_\forall.$$ This implies that $$||f(n)||_n = ||f(n+1)||_{n+1} = F_n(f(n)) = F_{n+1}(f(n+1))$$ for all $n$, and since $||f||_\forall = 1$, $1$ must be the common value.
It follows that $F_n$ attains its norm for all $n \in \N$.

We will also need two base Banach spaces, one for each of the truth values $\top$ and $\bot$:
\paragraph{\em The positive base}
Let $M_{\top}$ be the set of functions $f:\N \rightarrow \R$ such that $\sum_{n \in \N}|f(n)| < \infty$, let $||f||_{\top} = \sum_{n \in \N} |f(n)|$ and let $F_{\top}(f) = \sum_{n \in \N}f(n)$.
Then $F_\top$ has norm 1 and attains its norm. 

\paragraph{\em The negative base}
Let $M_{\bot}$ be the set of functions $f:\N \rightarrow \R$ such that $\lim_{n \rightarrow \infty}f(n) = 0$.
Let $||f||_{\bot} = \max\{|f(n)| \mid n \in \N\}$ and let $F_{\bot}(f) = \sum_{n \in \N}2^{-(n+1)}f(n)$.
Then $F_{\bot}$ has norm 1, but will not attain its norm.
 
\subsubsection{The Polish space of norms and functionals}
Throughout the rest of this section we will let $N$ be the vector-space over $\Q$ consisting of all functions from $\N$ to $\Q$ with finite support. We identify $N$ with the countable set of finite sequences $$(q_0, \ldots , q_{n-1})$$ from $\Q$ such that either $n = 0$ (the given sequence is empty) or $q_{n-1} \neq 0$. We let $\Theta$ be the zero-element of $N$, and depending on how we view $N$, $\Theta$ is either the constant zero function $0^\omega$ or the empty sequence.

We will restrict ourselves to norms $||\cdot||$ on $N$ with values in $\Q$,  and to linear functionals $F$ on $N$ with values in $\Q$. We may also identify $N$ with the set of rational linear combinations of free base vectors $v_0,v_1,v_2, \cdots.$ Then $\Theta$ will be the empty linear combination.

 Then the set $X$ of pairs $x = (||\cdot||_x,F_x)$, where $||\cdot||_x$ is a norm on $N$ and $F_x$ is a linear functional with norm bounded by 1, is a closed subset of $(N \rightarrow \Q)^2$ with the product topology, and is thus a Polish space. When we want more flexibility, we may consider all real valued norms on $N$, and the set of norms is still a Polish space. In the proof of Corollary \ref{2.10} we use a norm where the values on $N$ are algebraic numbers.

Given a norm $||\cdot||$ on $N$, the Banach space $M$ under consideration will be the completion of $N$ with respect to this norm, and the problem is then if (the canonical extension of) $F$ takes the value 1 somewhere on the unit sphere. Actually we will be dealing only with functionals with norm 1, and then the problem is if $F$ attains its norm. We let $Y$ be the set of $(||\cdot||,F) \in X$ where $F$ has norm 1 and attains its norm. It is the complexity of the set $Y$ that is under investigation in this section.

We will now  relate this way of representing separable Banach spaces and linear functionals to the constructions of the previous subsection.
\begin{obs} {\em The norms $ || \cdot ||_\top$  and $|| \cdot ||_\bot$, restricted to $N$, will take values in $\Q$. $M_\top$ and $M_\bot$ will, up to isomorphism of Banach spaces, be the completions of $N$ under these norms. }
\end{obs}
The constructions of $M_\exists$ and $M_\forall$ can be viewed as operators that for all infinite sequences $\{(M_n,F_n)\}_{n \in \N}$  produce the pairs $(M_\exists,F_\exists)$ and $(M_\forall,F_\forall)$.

We may then restrict these operators to pairs from $X$, and further only consider functions with finite support in the construction of the new vector spaces. Then the results will be the vector space over $\Q$ of functions from $\N^2$ to $\Q$ with finite support and with the restrictions to this space of the norms and linear functionals. Since $\N^2$ can be put in a 1-1-correspondence with $\N$ we may consider these operators as operators from $X^\N$ to $X$. In order to make our constructions uniform, we will use 
a fixed  computable bijection $\langle \cdot,\cdot \rangle$  between $\N^2$ and $\N$.
\begin{obs}{\em
Seen as operators from $X^\N$ to $X$, the operators
$$\{(M_n,F_n)\}_{n \in \N} \mapsto (M_\exists,F_\exists)$$
and
$$\{(M_n,F_n)\}_{n \in \N} \mapsto (M_\forall,F_\forall)$$
are continuous, and the operators commute with completions in the sense that if each $M_n$ is the completion of $(N,||\cdot||_n)$  then $M_\exists$ is the completion of $(N,||\cdot||_\exists)$ and $M_\forall$ is the completion of $(N,||\cdot||_\forall)$.}
\end{obs}
\subsubsection{The complexity of attaining the norm}
Let $X$ be as in the previous section.
Let $\mathcal X$ denote the class of all subsets $C$ of $\N^\N$ such that for some continuous function $\phi:\N^\N \rightarrow X$ we have that
\begin{center}$x \in C \Leftrightarrow F_{\phi(x)}$ attains its $||\cdot||_{\phi(x)}$-norm\end{center} or, in short,
$$x \in C \Leftrightarrow \phi(x) \in Y.$$
 We will show that all Borel sets in $\N^\N$ are in $\mathcal X$. This is an immediate consequence of the following:
\begin{lem}\hfill
\begin{enumerate}[\em a)]
\item
All subsets $C$ of $\N^\N$ that are both closed and open are in $\mathcal X$.

\item The class $\mathcal X$ is closed under countable unions and intersections.
\end{enumerate}
\end{lem}
\proof\hfill
\begin{enumerate}[a)]
\item We let $\phi(x) = (||\cdot||_{\top},F_{\top})$ if $x \in C$ and $\phi(x) = (||\cdot||_{\bot},F_{\bot})$ otherwise.
\item
For each $n \in \N$, let $C_n \in X$, and let $\phi_n$ define $C_n$ as in the definition of $\mathcal X$.
We let $\phi_\cup$ be the composition of $\{\phi_n\}_{n\in \N}$  with the restriction of the $M_\exists$-operator discussed above, and $\phi_\cap$ be the corresponding composition with the restriction of the $M_\forall$-operator.
Then $$x \in \bigcup_{n \in \N}C_n \Leftrightarrow \phi_\cup(x) \in Y$$ and $$x \in \bigcap_{n \in \N}C_n \Leftrightarrow \phi_\cap(x) \in Y.\eqno{\qEd}$$
\end{enumerate}

\begin{thm}\label{BO} Let $ Y$ be as above. Then $ Y$ is not Borel.
\end{thm}
\begin{proof}
There is no Borel set $B$ such that all other Borel sets $C$ are the continuous preimages of $B$. This is a standard result in descriptive set theory, see Kechris \cite{Kechris} for an introduction.\end{proof}
\begin{rem} As mentioned above, this theorem is a consequence of a theorem due to R. Kaufman \cite{Kauf}. In order to employ Kaufman's result, note that the unit ball of the dual of a separable Banach space is compact and metrizable in the weak star topology, and that the Borel structure for the unit ball will be the same for the weak star topology and the topology we consider.\end{rem}
\subsubsection{Closed balls }
Now, let us turn the attention to the original problem. Let $Z$ be the set of norms on $N$, considered as a closed subset of $N \rightarrow \Q$. 
\begin{cor}\label{Cor1}
The set of norms $||\cdot|| \in Z$ such that the balls around \begin{center}$(2,0,\Theta)$, $(0,2,\Theta)$ and $(2,2,\Theta)$\end{center} with radius 1 have a nonempty intersection in the completion, will be analytic, but not Borel.\end{cor}
This is a direct consequence of Lemma \ref{AE}, Theorem \ref{BO} and the fact that the construction in Lemma \ref{AE} can be seen as a continuous map from $X$ to $Z$.
\subsection{Effective versions}
All constructions underlying our proof of Theorem \ref{BO} are computable. This leads us to the following consequence of the proof:
\begin{cor}\label{cor.eff} Let $C \subseteq \N$ be hyperarithmetical. Then there is a computable map
$$n \mapsto ||\cdot||_n$$ from $\N$ to the set of norms on $N$ such that the three unit closed balls around $(2,0,\Theta)$, $(0,2,\Theta)$ and $(2,2,\Theta)$ have nonempty intersection in the $||\cdot||_n$ completion of $N$if and only if $n \in C$.
\end{cor}
\begin{proof}
When $C$ is hyperarithmetical, there is a computable, well founded labelled tree $T$ on $\N$ such that
\begin{enumerate}[i)]
\item leaf nodes are labelled with an index for a primitive recursive set.
\item branching nodes are labelled with $\cap$ or with $\cup$
\end{enumerate}
and such that $C_{\rm root} = C$. (For each node $\sigma$ in $T$, $C_\sigma$ is defined by recursion on $\sigma$ according to the labeling.)
Then we use the recursion theorem together with the codings of
$\top$,$\bot$, $\exists$ and $\forall$ as operators on elements in
$X$. For each node $\sigma \in T$ this leaves us in a computable way
with a norm $||\cdot||_{\sigma,n}$ and a linear functional
$F_{\sigma,n}$ such that $F_{\sigma,n}$ attains its norm if and only
if $n \in C_\sigma$. Finally, we use Lemma \ref{AE}, and the corollary
is proved.\end{proof}

We also obtain the following corollary of our proof:
\begin{cor}\label{2.10}
For each hyperarithmetical set $C$ there is a computable norm on $N$ such that $C$ is $m$-reducible to the set of triples of vectors in $N$ such that the three corresponding closed balls with radius 1 have a nonempty intersection in the completion.\end{cor}
\begin{proof}
For each $n$, let  $M_n$ be the Banach space realizing Corollary \ref{cor.eff} for $C$ and $n$, and let $B_1$, $B_2$ and $B_3$ be the three closed balls used.
We consider the {\em Euclidian product} of the $M_n$'s based on the norm $$||f|| = (\sum_{n \in \N}||f(n)||_n^2)^\frac{1}{2}.$$ We obtain one computable Banach space with a computable  list of triples of closed balls $$\{(B_{n,1},B_{n,2},B_{n,3})\mid n \in \N\}$$such that for all $n$, $n \in C$ if and only if $$B_{n,1} \cap B_{n,2} \cap B_{n,3} \neq \emptyset.$$
We obtain this by placing one copy of the centers of $B_1$, $B_2$ and $B_3$ in each coordinate, and observe that if three balls with centers in one coordinate intersect anywhere, then they intersect in that coordinate.\end{proof}
\begin{rem} The results of Kaufman \cite{Kauf} suggest that there is a computable Banach space where the set of triples of closed balls with radius 1 and centers in the given dense set is complete $\Sigma^1_1$, but our proof method does not give us that.\end{rem}

\section{Internal algorithms on Banach spaces}\label{Sec3}
\subsection{Discussion} If $X$ is a computable metric space, the internal structure will consist of the given enumeration of a dense subset together with the metric. In order to form a language in which we can express what we may call {\em internal algorithms} we need some prime total or partial  operators that make sense for all separable metric spaces, that are computable and that can be combined with computable functions on $\N$, constructions by recursion and a selected set of computable functions and operators on $\R$, in order to define other functions. One possible candidate could be the function $lim$ that to any fast converging sequence gives the limit, and that is undefined when the argument is not fast convergent. One problem with this will be that a metric space in general is not an algebraic structure in any sense, so it is difficult internally to construct sequences that we might apply $lim$ to in order to make use of it.

A separable Banach space has a much richer internal structure. In this section, we will suggest two computable operators, one total and one partial, and we will discuss the internal computational power of these two operators. Using the first operator, which is a total extension of the partial function giving the limit of a fast converging sequence, we will show that if a Banach space $M$ is computably closed as a subspace of a metric space $X$, then it is a computable retract of that space, see Subsection 3.3. We will use this in Subsection 3.4 in order to prove an effective version of the embedding theorem from \cite{D3}. In Subsection 3.5 we will define the {\em accumulation operator} $\bar \alpha$, and show that $\bar \alpha$ can be used to ``write programs" for every computable function from a computable metric space to a computable Banach space.

We will employ definition by cases, also adjusted to $\R$, and definitions by recursion. We consider $\R$ as a given Banach space, and we consider ``norm", ``vector sum" and ``multiplication with a scalar" for the given Banach spaces as basic computable functions.

\subsection{Partial functions and definitions by cases}
We have defined what we mean by a computable function $f:X \rightarrow Y$, when $X$ and $Y$ are computable metric spaces. There are several natural ways to extend this to partial functions. We will choose the one that is most useful for the purpose of this paper.

\begin{defi}Let $X$ and $Y$ be computable metric spaces, $Z \subseteq X$ open.
Let $f:Z \rightarrow Y$. We say that $f$ is {\em partial (externally) computable} if there is a partial computable $F:\N^\N \rightarrow \N^\N$ such that $F(\gamma)$ is defined and represents $f(x)$ whenever $\gamma$ represents $x \in Z$.\end{defi}
Let $Seq$ be the set of finite sequences from $\N$ ordered by the initial segment ordering $\prec$. We identify $Seq$ with the set of sequence numbers.

If $F:\N^\N \rightarrow \N^\N$ is partial computable in the Turing-Kleene sense, there will be a computable function $\hat F:Seq \rightarrow Seq$ that is $\prec$-increasing and such that
$$F(\gamma) = \lim_{n \rinf} \hat F((\gamma(0) , \ldots , \gamma(n-1))).$$
Thus every partial externally computable  function from $X$ to $Y$ is represented by a computable $\hat F$.
\begin{defi}Let $f$ and $g$ be partial functions from $\R$ to any set $X$.

We define the function $Case(f,g,y)$ by
\[Case(f,g,y)(x) = \begin{cases}
  f(x)&\mbox{when $x < y$;}\\
  g(x)&\mbox{when $y < x$;}\\
  z&\mbox{if $f(y) = g(y) = z$;}\\
  \mbox{undefined}&\mbox{if $f(y) \neq g(y)$.}
\end{cases}
\]
\end{defi}
The following is a trivial observation:
\begin{lem} If $X$ is a computable metric space, and $f$ and $g$ are partial externally computable functions from $\R$ to $X$, then
$$h(y,x) = Case(f,g,y)(x)$$ is externally computable.\end{lem}
The Case operator will eventually be one of the operators generating the internally computable functions, see Definition \ref{def.int.comp}. Using $Case$ we can define the functions $\min$ and $\max$ on $\R^2$,  so they will be internally computable.
\subsection{Modified limits}
In this subsection we will let $M$ be a fixed separable Banach space with norm $||\cdot||$, and we use standard vector space notation for $M$.
\begin{defi} Let $\{w_n\}_{n \in \N}$ be a sequence from $M$.
\begin{enumerate}[a)]
\item We let the {\em modified sequence} $\{w^m_n\}_{n \in \N}$ be defined by 
\begin{enumerate}[$-$]
\item $w^m_0 = w_0$
\item $w^m_{k+1} = w^m_k + \lambda_k(w_{k+1} - w^m_k)$
 where $\lambda_k = 1$ if $||w_{k+1} - w^m_{k}|| \leq 2^{-k}$ and $$\lambda_k = \frac{2^{-k}}{||w_{k+1} - w^m_k||}$$ if $||w_{k+1} - w^m_k|| \geq 2^{-k}$.
\end{enumerate}
\item The {\em modified limit} of $\{w_k\}_{k \in \N}$ is

$$ml(\{w_k\}_{k \in \N}) = \lim_{k \rinf} w^m_k.$$
\end{enumerate}\end{defi}
\begin{lem}\hfill
\begin{enumerate}[\em a)]
\item The maps $$\{w_k\}_{k \in \N} \mapsto \{w^m_k\}_{k \in \N}\quad\hbox{and}\quad\{w_k\}_{k \in \N} \mapsto ml(\{w_k\}_{k \in \N})$$ are total and computable.
\item If we for all $k$ have that $||w_k - w_{k+1}|| \leq 2^{-k}$, then $w^m_k = w_k$ for all $k$.
\end{enumerate}\end{lem}\medskip

\noindent This is trivial and can safely be left for the reader.

When we  add the modified limit to our toolbox of internally computable functions, it is simply in order to replace the partial computable function giving the limit of a fast convergent sequence with a total one. All applications we have in mind involve constructing a fast converging sequence and then consider taking the limit as an internal process.
The map $$\{w_k\}_{k \in \N} \mapsto \{w^m_k\}_{k \in \N}$$ is a retraction of the set of all sequences to the set of fast converging sequences. 
\begin{rem}\label{mod}
If $f:(0,1) \rightarrow (0,1)$ is continuous and strictly monotone increasing, and $y$ is in the range of $f$, we can use the Case operator and modified limits over $\R$ to compute $x$ from $f$ and $y$ such that $f(x) = y$.

Since every open interval with computable endpoints, the possibilities of $\pm \infty$ included, can be put in a computable, strictly monotone correspondence with $(0,1)$, we use this to claim that the inverse of a strictly monotone continuous function $f$ defined on an interval is computable in $f$.\end{rem}
\begin{defi} Let $X$ be a computable metric space, $Y \subseteq X$ be a closed subset.
We say that $Y$ is {\em computably closed} in $X$ if
\begin{enumerate}[$-$]
\item $Y$ is the closure of a computably enumerable subset of $X$ and
\item the function giving the distance from $x \in X$ to $Y$ is computable.
\end{enumerate}\end{defi}
\begin{thm}\label{theorem.retract} Let $X$ be a separable metric space, $M \subset X$ a closed subspace such that $M$ is also a Banach space with the induced metric.
 There is a continuous function $g:X \rightarrow M$ such that $g$ is the identity on $M$ and such that
 $g$ is internally computable in the 
 enumeration of a dense subset in $M$
, the metric on $X$
and the distance function $d(x,M)$.

\end{thm}
\begin{proof}
We will first define $g$. Then we will use elementary means to define a sequence $\{g_n\}_{n \in \N}$ such that composition with modified limits gives us $g$ as the limit.

We use the machinery of {\em probabilistic projections} from \cite{D3}, but our treatment is self contained.

For nonnegative reals $x$ and $y$ we let $$x \dminus y = \max\{x-y,0\}.$$
Let $\{a_n\}_{n \in \N}$ be an enumeration of a dense subset of $M$ and let $M_n = \{a_0 , \ldots , a_n\}$.
For each $n \in \N$ and each $x \in X$ we define the probability distribution $\mu_{n,x}$ on $M_n$ by
$$\mu_{n,x}(a) = \frac{(d(x,M_n) + 2^{-n}) \dminus d(x,a)}{\sum_{b \in M_n}((d(x,M_n) + 2^{-n}) \dminus d(x,b))}$$ and we let $$f_n(x) = \sum_{a \in M_n}\mu_{n,x}(a)\cdot a.$$
If $y = n + \lambda$ where $\lambda \in [0,1)$, we let
$$d(x,M_y) = (1-\lambda)d(x,M_n) + \lambda d(x,M_{n+1})$$ and we let $$f_y(x) = (1-\lambda)f_n(x) + \lambda f_{n+1}(x).$$
(We have not defined $M_y$, only interpreted the expression $d(x,M_y)$.)
We let 
\[
g(x)=\begin{cases} x&\mbox{\phantom{where $d(x,M) = 2^{-y}$} if $x \in M$}\\
              f_y(x)&\mbox{where $d(x,M) = 2^{-y}$ if $x \not \in M$.}
     \end{cases}
\]
We will show that $g$ is the limit of the fast converging (and thereby uniformly converging) sequence $\{g_n\}_{n \in \N}$, that in addition will be internally computable. Thus $g$ itself will be internally computable.

We have designed the probability distributions $\mu_{n,x}$ in such a way that if $a$ is an element in $M_n$ with minimal distance to $x$, and $b$ is further away from $x$ than $2^{-n} + d(a,x)$, then $\mu_{n,x}(b) = 0$. If, in addition $d(a,x) \leq 2^{-n}$, we know from the triangle inequality that if $\mu_{n,x}(b) > 0$ and $\mu_{n,x}(c) > 0$, then $||b - c|| \leq 2^{2-n}$. We will use this observation when verifying that $\{g_n\}_{n \in \N}$ is fast converging.
Let $$h_x(y) = d(x,M_y) + 2^{-(y+4)}.$$
Then $h_x$ is strictly decreasing, and $$d(x,M) = \lim_{y \rinf}h_x(y).$$
Let $n$ and $x$ be fixed. Let $y = y_{n,x}$ be minimal such that $d(x,M) = 2^{-y}$ or such that $$(h_x(y-1) \leq 2^{-(n+3)}) \wedge (y \geq n+3).$$
If $d(x,M) > 0$ there will be some $y$ satisfying the first option while if $d(x,M) = 0$ there is a $y$ that will satisfy the second option. For each option, $y$ is unique.\bigskip

\noindent{\em Claim.}\newline
$y_{n,x}$ is internally computable from $x$ and $n$.\medskip

\noindent{\em Proof of Claim.}\newline 
We use the computability of inverse functions (see Remark \ref{mod}) and a nested application of the Case operator.
We split between the cases $d(x,M) \geq 2^{-(n+5)}$ and $d(x,M) \leq 2^{-(n+5)}$.
If $d(x,M) \geq 2^{-(n+5)}$, we find $y \leq n+5$ such that  $d(x,M) = 2^{-y}$ and then  check, formally using the Case operator, between the two subcases
\begin{enumerate}[(1)]
\item $(h_x(y-1) \leq 2^{-(n+3)}) \wedge (y \geq n+3)$
\item $(h_x(y-1) \geq 2^{-(n+3)}) \vee (y \leq n+3)$
\end{enumerate}
In subcase (2) we have that $y_{n,x} = y$ while in subcase (1) we use that $h_x$ is strictly monotone in order to compute $y_{n,x}$.

If $d(x,M) \leq 2^{-(n+5)}$ we know that there is a $y$ such that $h_x(y-1) \leq 2^{-(n+3)}$, and we can computably find that $y$. This information can be used to isolate the exact value of $y_{n,x}$ in analogy with the first case.
This ends the proof of the claim.

Let $$g_n(x) = f_{y_{n,x}}(x).$$
$g_n$ is clearly  internally computable uniformly in $n$. We will prove that the sequence $\{g_n(x)\}_{n \in \N}$ is fast converging. 

Let $n \in \N$ and $x \in X$ be fixed. Let $y = y_{n,x}$.
If $d(x,M) = 2^{-y}$ then $y = y_{m,x}$ when $n \leq m$, and $g_n(x) = g_m(x)$ for $m \geq n$.
If $d(x,M) < 2^{-y}$ we have that $$m \geq n \Rightarrow y_{m,x} \geq y_{n,x}.$$
Let $y = k + \lambda$ where $\lambda \in [0,1)$. Since
$$h_x(y-1) \leq 2^{-(n+3)}$$ we in particular have that
$$d(x,M_{y-1}) \leq 2^{-(n+3)}$$  and we have that$$d(x,M_{k+1}) \leq d(x,M_k) \leq d(x,M_{y-1}) \leq 2^{-(n+3)}.$$
Since $y \geq n+3$, we have that $k > n+2$. Then,  if $\mu_{k,x}(b) > 0$ we have that $d(x,b) < 2^{-(n+1)}$, and thus $f_y(x)$ can be seen as a convex combination of vectors that have distances $\leq 2^{-(n+1)}$ to $x$.

The same argument is valid for $f_{y_{m,x}}(x)$ when $n \leq m$. Thus, any two vectors in the two sums with positive coefficients will have distance $\leq 2^{-n}$. It follows that $$||g_n(x) - g_m(x)|| \leq 2^{-n}$$ when $m \geq n$.\end{proof}
\subsection{An effective embedding theorem}
The Urysohn space $U$ was introduced by P. Urysohn \cite{U1,U2}. $U$ is, by construction, a computable metric space. $U$ is characterized up to isometry by being complete, separable and satisfying the following homogeneity property:
\begin{enumerate}[$-$]
\item If $A \subseteq B$ are two finite metric spaces and $f:A \rightarrow U$ is an isometric embedding, then $f$ can be extended to an isometric embedding $g:B \rightarrow U$.
\end{enumerate}\medskip

\noindent A model theorist might use the expression {\em $\omega$-saturated} for this property.
The classical construction of $U$ is carried out in two steps:
\begin{enumerate}[$-$]
\item We construct a computable metric $d_0$ on $\N$ with values in
  $\Q$ such that this space satisfies the homogeneity property with
  respect to $\Q$-valued metric spaces. Let $U_0$ be this space. We
  say that $U_0$ is {\em $\Q$-homogenous}.\medskip

\item We let $(U,d)$ be the completion of $U_0$ and prove the homogeneity property in full generality.
\end{enumerate}\medskip

\noindent Le\v{s}nik \cite{DL} gave an analysis of the Urysohn space from the point of view of \emph{constructive analysis}, and he proved that embeddability of separable metric spaces into the Urysohn space holds constructivly. Our next lemma is a related result for computational analysis, making it a point that the image of the embedding is computationally closed.
\begin{lem}\label{Ur}
Let $X$ be a computable metric space. Then $X$ is computably isometric to a computably closed subset of $U$.
\end{lem}
\begin{proof}
Without loss of generality, we assume that $X $ is the completion of $ \N$ with a computable pseudometric $d$ on $\N$, meaning that $d$ satisfies symmetry and the triangle inequality, but not necessarily that $n \neq m$ implies that $d(n,m) > 0$. In order to avoid confusing terminology, we let $U$ be the completion of the $\Q$-homogeneous and  computable metric on $\{u_k \mid k \in \N\}$.
\medskip

\noindent {\em Claim 1}\newline
Let $k_1 , \ldots , k_n \in \N$ and $r_1 , \ldots , r_n \in \Q_{\geq 0}$. Then $$\{u_k \mid \bigwedge_{i = 1}^n d_U(u_k,u_{k_i}) \geq r_i\}$$ is also $\Q$-homogeneous, and computably isometric to $U_0$.

The first part is implicit in Urysohn's original paper, and the second part follows by a straightforward back-and-forth construction.
Of course, $r_i = 0$ represents no restriction, and we include this case only for expositional reasons.
\medskip

\noindent{\em Claim 2}\newline
Let $A \subseteq B$ be finite, computable metric spaces and let $f:A \rightarrow U$ be a computable embedding. Then, uniformly in $A$, $B$ and $f$ there is a computable embedding $g:B \rightarrow U$ that extends $A$.

In order to prove this claim, it is sufficient to do so for one-step extensions. Then the original proof of Urysohn is sufficiently constructive to give us this.

We construct an isometric map $f:(\N,d_0) \rightarrow U$ by recursion on $n$, and simultaneously we identify for each $n$ a radius $r_{i,n} \in \Q$ for $i = 0 , \ldots , n$ protecting a neighborhood of $u_i$ from being hit by the completion of $f$, as follows:
\begin{enumerate}[$-$]
\item Assume that $f(i)$ is defined as the limit of the fast converging sequence $\{v_{i,m}\}_{m \in \N}$ from $U_0$ for $i = 0 , \ldots , n-1$.
For $j = 0 , \ldots , n$ we let $r_{j,n} = k\cdot2^{-n}$ where $k$ is maximal such that 
$$d(u_j,v_{i,n}) \geq (k+2)\cdot2^{-n}$$ for all $i = 0 , \ldots , n-1$, if there is such $k \in \N$. If there is no such $k$, we let $r_{j,n} = 0$.
Then we also know   that $d(u_j,v_{i,m}) \geq r_{j,n}$ when $j \leq
n-1$ for $m > n$.\medskip

\item We use the algorithm from Claim 2 combined with Claim 1 to find $f(n)$ in $U$ as the limit of the fast converging sequence $\{v_{n,m}\}_{m \in \N}$  such that $d_U(v_{n,m},u_j) \geq r_{j,n}$ for all $j = 0 , \ldots , n$ and $m \in \N$.
\end{enumerate}
Using the recursion theorem, these constructions can be unified into one computable construction.

Now, $f$ extends to an isometric map $g:X \rightarrow U$.
In this construction, at a step $m \geq n$ we approximate the distance from $u_n$ to the image of $g$ with a precision of $2^{-(m-2)}$. Thus the image of $g$ will be computably closed. \end{proof}
We may draw a simple consequence out of Theorem \ref{theorem.retract} and Lemma \ref{Ur}:
\begin{cor}
Let $\sigma(\vec X)$ be a finite type in the type variables $\vec X$ and let $\vec M$ be a sequence of computable Banach spaces. Let $U$ be the Urysohn space. Then $\sigma(\vec M)$ is a retract of $\sigma(\vec U)$, where the embedding and inverse both are internally computable.\end{cor}
\begin{proof}
If $\sigma$ is one of the base types, this is a consequence of Lemma \ref{Ur} and Theorem \ref{theorem.retract}, while if $\sigma = \delta \rightarrow \tau$ this follows immediately from the assumption that the corollary holds for $\delta$ and $\tau$.\end{proof}
\begin{rem}A similar result was proved in \cite{D3}. There we replace $\vec M$ with a sequence $\vec A$ of complete, separable metric spaces, but we then only obtain topological embeddings for each type, and even if $\vec A$ are computable, these embeddings will not be computable in general.\end{rem}

\subsection{The accumulation operator}
In this subsection, we will introduce one computable operator $\bar \alpha$ on Banach spaces. We call it the {\em accumulation operator}, and we think of it as a generalization of the $\mu$-operator. Like the $\mu$-operator, $\bar \alpha$ will be partial. We will prove that every externally computable function from a computable metric space to a computable Banach space will be internally computable relative to $\bar \alpha$.
\begin{defi}
Let $M$ be a Banach space, let $f:\N \rightarrow \R_{\geq 0}$ and $g:\N \rightarrow M$.
We define $\bar \alpha(f,g) = v$ if for some $n$ and $\lambda$ we have that
$$\sum_{i = 0}^{n-1}f(i) \leq 1, \;\;\sum_{i = 0}^nf(i) > 1,\;\;\lambda = 1-\sum_{i = 0}^{n-1}f(i)$$
and $$v = \sum_{i = 0}^{n-1}f(i) g(i) + \lambda  g(n).$$

\noindent We let $\bar \alpha(f,g) = g(0)$ if $f(0) > 1$.
We let $\bar \alpha(f,g)$ be undefined if $\sum_{i = 0}^\infty f(i) \leq 1$.
We call $\bar \alpha$ {\em the accumulation operator}.
\end{defi}

We will let $\bar \alpha(f,g) \!\!\downarrow$ mean that $\bar \alpha (f,g) = v$ for some $v \in M$, and $\bar \alpha (f,g) \!\!\uparrow$ mean that it is undefined.
\begin{lem} \label{alphacont} If $f = \lim_{k \rinf}f_k$, $g = \lim_{k \rinf}g_k$, $\bar \alpha(f,g)\!\!\downarrow$ and $\bar \alpha(f_k,g_k)\!\!\downarrow$ for each $k \in \N$, we have that $$\bar \alpha(f,g) = \lim_{k \rinf}\bar \alpha(f_k,g_k).$$
\end{lem}
\proof
Let $n$, $n_k$, $\lambda$ and $\lambda_k$ be as in the definition of $\bar \alpha$.
Then $$\sum_{i = 0}^n f(i) > 1$$ so for some $k_0$ we have that $\sum_{i = 0}^n f_k(i) > 0$ whenever $k \geq k_0$.
Then $n_k \leq n$ for all $k \geq k_0$, and we may, without loss of generality, assume that $n_k$ takes a fixed value $m \leq n$ for all $k$. The proof now splits into two cases:\medskip

\noindent{\em Case 1}\newline
 $m = n$.
Then $\lambda = \lim_{k \rinf}\lambda_k$ and  clearly $\bar \alpha(f,g) = \lim_{k \rinf}\bar \alpha(f_k,g_k)$.\medskip

\noindent{\em Case 2} \newline
$m < n$.
Then $\sum_{i = 0}^m f_k(i) > 1$ for all $k$, and $\sum_{i = 0}^m f(i) \leq 1$, so $\sum_{i = 0}^mf(i) = 1$.
It follows that $$f(m+1) = \cdots = f(n-1) = 0$$ and that $\lambda = 0$. Then
$$\bar \alpha(f,g) = \sum_{i = 0}^mf(i)\cdot g(i) = \lim_{k \rinf}\sum_{i = 0}^m f_k(i)\cdot g_k(i).$$
Since we also must have that $f(m) = \lim_{k \rinf}\lambda_k$ and that 
$f(m) = \lim_{k \rinf}f_k(m)$, we have that
$$\eqalign{
\lim_{k \rinf}\bar \alpha(f_k,g_k) 
&= \lim_{k \rinf}(\sum_{i =  0}^{m-1}f_k(i)\cdot g_k(i) + \lambda_k g_k(m))\cr
&= \lim_{k \rinf}(\sum_{i = 0}^{m-1}f_k(i)\cdot g_k(i) + f_k(m) \cdot
g_k(m)) = \bar \alpha(f,g).\rlap{\hbox to 61 pt{\hfill\qEd}}}$$

\begin{lem}\label{alpha.comp} $\bar \alpha$ is externally computable.
\end{lem}
\begin{proof}
Let $k \in \N$ and let $\hat f$ and $\hat g$ be external representatives for $f$ and $g$ resp. We will describe an algorithm that use $\hat f$ and $\hat g$ as oracles and that decides the value of $\bar \alpha(f,g)$ with a precision of $2^{-k}$.

First we use $\hat f$ and find $\hat n \in \N$ such that
$$\sum_{i = 0}^{\hat n}f(i) > 1.$$
Let $n$ and $\lambda$ be as in the definition of $\bar \alpha(f,g)$. Then we know that $n \leq \hat n$, but we cannot in general compute $n$ from $\hat f$. In order to compute $\bar \alpha(f,g)$ we will run two algorithms in parallel, both giving a valid answer when terminating, and choose the algorithm terminating first. In all cases, at least one of the two will terminate.
\medskip

\noindent{\em Algorithm 1}\newline
This algorithm will work when $n = 0$ as well as when $\sum_{i = 0}^{n-1}f(i) < 1$ and $\sum_{i = 0}^nf(i) > 1$.

We search for an $n \leq \hat n$ and verifications of  $\sum_{i = 0}^{n-1}f(i) < 1$ and $\sum_{i = 0}^nf(i)> 1$ and then compute
$$\sum_{i = 0}^{n-1}f(i)\cdot g(i) + (1-\sum_{i = 0}^{n-1}f(i))\cdot g(n)$$with the desired precision. This can be done using $\hat f$ and $\hat g$.
\medskip

\noindent{\em Algorithm 2}\newline
This will work when $n > 0$ and $\sum_{i = 0}^{n-1}f(i) = 1$.

First we use $\hat g$ to find an integer $N$ that will be an upper bound for $$\{||g(0)||, \ldots , ||g(\hat n)||\}.$$
We then find $m_1$ and $n_1$ such that $m_1 < n_1 \leq \hat n$ and such that
\begin{enumerate}[i)]
\item $m_1 = 0$ or $\sum_{i = 0}^{m_1 - 1}f(i) < 1$

\item $\sum_{i = 0}^{m_1}f(i) \in (1 - \frac{2^{-(k+3)}}{N},1+\frac{2^{-(k+3)}}{N})$

\item $\sum_{i = m_1}^{n_1 - 1}f(i) < \frac{2^{-(k+3)}}{N}$

\item $\sum_{i = 0}^{n_1}f(i) > 1.$

\end{enumerate}
Then we know that $$m_1 < n < n_1.$$
The point is that if we make the educated guess that $n = m_1 $, then the error we make is bounded by $2^{-(k+1)}$. Then, when we compute $$\sum_{i = 0}^{m_1 - 1}f(i) \cdot g(i) + (1 - \sum_{i = 0}^{m_1 - 1} f(i))\cdot g(m_1)$$ with a precision of $2^{-(k+1)}$ we have computed $\bar \alpha(f,g)$ with a precision of $2^{-k}$.
\medskip

This ends our description of the algorithm and the proof of the lemma.\end{proof}
Given computable metric spaces $X_1 , \ldots , X_n$ and computable Banach spaces $M_1, \ldots , M_m$ we may now give a precise definition of what we mean by being {\em internally computable}. In order to avoid too much notation, we restrict the definition to partial functions between cartesian products of four spaces, $\N$, $\R$, $X$ and $M$ where $X$ is an arbitrary computable metric space and $M$ is an arbitrary computable Banach space.

\begin{defi} \label{def.int.comp}
Let $(X, d_X)$ be a metric space given as the completion of the metric
on $\{x_n\mid n \in \N\}$ and let $(M,||\cdot||_M, +_M , \cdot_M)$ be
a separable Banach space that considered as a metric space is the
completion of the metric on $\{y_m\mid m \in \N\}$.

The set of {\em internally computable functions} from any product of the sets $\N$, $\R$, $X$ and $M$ to any of these sets is the smallest class of partial functions containing the basic functions and constants
\begin{enumerate}[$-$]
\item successor and constants on $\N$
\item the constant 1 on $\R$
\item addition,  subtraction, multiplication and the absolute value function on $\R$
\item $n \mapsto x_n$ and $d_X:X^2 \rightarrow \R$
\item $m \mapsto y_m$, $+_M:M^2 \rightarrow M$, $\cdot_M:\R \times M \rightarrow M$ and $||\cdot||_M:M \rightarrow \R$
\end{enumerate} and is closed under 
\begin{enumerate}[$-$]
\item composition
\item primitive recursion \item $\mu$-recursion 
($\mu$-recursion only makes sense over $\N$ but primitive recursion can be used to define infinite sequences in any of these sets)\item the Case operator
\item the modified limit operator for $\R$ and $M$\item the accumulation operator for $\R$ and $M$\end{enumerate}
where we allow for dummy arguments of all kinds in the basic functions.
\end{defi}

\begin{thm}\label{3.16}Let $X$ be a computable metric space and let $M$ be a computable Banach space. Let $f:X \rightarrow M$ be total and externally computable.
Then $f$ is internally computable.
\end{thm}
\proof
We will denote the distance, norm and the algebraic functions on $M$, $\R$ and $\N$ without indices, since which we use will always be clear from the context.

Recall that $(Sec,\prec)$ is the set of (sequence numbers of) finite sequences of integers with the standard initial segment ordering.

Let $\hat F:Sec \rightarrow Sec$ be computable and monotone and let $F$ be the partial (in the Turing sense) computable function from $\N^\N \rightarrow \N^\N$ defined by
$$F(\gamma) = \lim_{n \rinf}\hat F((\gamma(0), \ldots , \gamma(n-1)))$$ such that whenever $\gamma:\N \rightarrow \N$ is total and $\{x_{\gamma(n)}\}_{n \in \N}$ is a fast converging sequence with limit $x \in X$, then $\xi = F(\gamma)$ is total and $\{y_{\xi(m)}\}_{m \in \N}$ is a fast converging sequence with limit $f(x) \in M$.

Uniformly in $k$ we will construct $f_k$ as an internally computable function in such a way that $f(x) = \lim_{k \rinf}f_k(x)$ and the sequence is fast converging. We then compose this sequence with the modified limit operator.

Let $k \geq 0$ be fixed, and let
$$\Delta_k = \{\sigma \in Seq \mid lh(\hat F(\sigma)) \geq k\}.$$
For each $\sigma = (s_0 , \ldots , s_{n-1}) \in \Delta_k$, $i < n$ and
$x \in X$ we let 
\[\phi_{\sigma,i}(x) = 
\begin{cases}
 1&\mbox{if $d(x,s_i) \leq 2^{-(i + 3)}$,}\\
 0&\mbox{if $d(x,s_i) \geq 2^{-(i+2)}$, and}\\
 \lambda&\mbox{if $0 \leq \lambda \leq 1$ and $d(x,s_i) =
   2^{-(i+3)}(1+\lambda)$.}
\end{cases}
\]
and we let $$\Phi_{\sigma}(x) = \prod_{i < lh(\sigma)}\phi_{\sigma,i}.$$\medskip

\noindent{\em Claim 1}\newline
If $\Phi_\sigma (x) > 0$ and $lh(\sigma) = n$, then there is a continuation of the sequence $$x_{\sigma(0)}, \ldots ,x_{\sigma(n-1)}$$ to a fast converging sequence with $x$ as the limit.

The proof is trivial, and is left for the reader.
\medskip

\noindent{\em Claim 2}\newline
For every $x \in X$ there is a $\sigma \in \Delta_k$ with $\Phi_\sigma(x) = 1$.
\medskip

\noindent{\em Proof of Claim 2}\newline
Choose $\gamma:\N \rightarrow \N$ such that
$$d(x,\gamma(i)) \leq 2^{-(i+3)}$$ for every $i \in \N$.
Then $F(\gamma)$ is total, so for some $n$, $\hat F((\gamma(0), \ldots , \gamma(n-1)))$ will have length $\geq k$ and  $\sigma = (\gamma(0), \ldots , \gamma(n-1)) \in \Delta_k$.
By the  construction of $\gamma$ we have that $\Phi_\sigma(x) = 1$.

Remark: From the argument, we see that there will be more than one such $\sigma$.

Now we enumerate $\Delta_k$ in a computable way as $$\Delta_k = \{\sigma_n\mid n \in \N\}.$$
Let 
\begin{enumerate}[$-$]
\item $h_x(n) = \Phi_{\sigma_n}(x)$;
\item $\hat F(\sigma_n) = \tau_n = (t_{n,0}, \ldots , t_{n,m_n - 1})  $ and let $g(n) = y_{t_{n,m_n-1}}$;
\item $f_k(x) = \bar \alpha(h_x,g)$.
\end{enumerate}
From the remark in the proof of Claim 2 we see that $\bar \alpha(h_x,g)$ will take a value.

By definition, $\bar \alpha (h_x,g)$ will be a sum
$$\lambda_0\cdot y_{g(0)} +\cdots + \lambda_m \cdot y_{g(m)}$$ for some $m$ where $$\sum_{i = 0}^m \lambda_i = 1$$ and $\lambda_i \leq h_x(i)$ for all $i$ (We have equality for all but at most one $i$, namely $i = m$.)
If $\lambda_i > 0$, we have that $h_x(i) > 0$, i.e. $\Phi_{\sigma_i}(x) > 0$.

By Claim 1 and the properties of $\hat F$, $\hat F(\sigma_i) = \tau_i$ can be extended to a fast converging sequence with $f(x)$ as a limit. It follows that the distance from $f(x)$ to $g(i) = y_{t_{i,m_i - 1}}$ is bounded by $2^{-k}$ (since the length of $\tau_i$ is at least $k$).

But then, by a direct application of the triangle inequality in $M$ we have that $$||f(x) - f_k(x)|| \leq 2^{-k}.\eqno{\qEd}$$
                            
\section{Epilogue}
In the previous section we showed that an externally computable function from a computable metric space $X$ to a computable Banach space $M$ in a reasonable sense is internally computable. Let us consider externally computable functions from $X$ to another computable metric space $Y$ instead. We will analyze to what extent the proof of Theorem \ref{3.16} can be adjusted to this more general case and isolate externally computable operators that will generate any externally computable $f$. 
\begin{defi} Let $Y$ be a computable metric space, and let $\{y_n \mid n \in \N\}$ be the dense subset with a computable metric defining $Y$ as a computable space.
\begin{enumerate}[a)]
\item A {\em probability distribution} on $\N$ is a map $\mu:\N \rightarrow \R_{\geq 0}$ such that 
$$\sum_{n \in \N}\mu(n) = 1.$$
\item A sequence $\{\mu_k\}_{k \in \N}$ of probability distributions on $\N$ is {\em fast converging} in $Y$ if, whenever $\gamma:\N \rightarrow \N$ is such that $$\forall k \in \N(\mu_k(\gamma (k)) > 0)$$ then $\{y_{\gamma (k)}\}_{k \in \N}$ is fast converging in $Y$.
\end{enumerate}\end{defi}
\begin{obs}{\em If $\{\mu_k\}_{k \in \N}$ is fast  converging in $Y$,  the limit of the sequence $\{y_{\gamma (k)}\}_{k \in \N}$ is independent of $\gamma$ as long as $\mu_k(\gamma(k)) > 0$ for all $k$. Moreover, this limit is externally computable from the sequence $\{\mu_{k}\}_{k \in \N}$.}\end{obs}
There does not seem to be an analogue of the modified limit operator in this general  setting, so we will have to content ourselves with the operator that gives the limit of any sequence of probability distributions that is fast converging in $Y$.
\newline
We will now consider an analogue $\alpha^\ast$ of the accumulation operator $\bar \alpha$:
\begin{defi} Let $f:\N \rightarrow \R_{\geq 0}$ and let $g:\N \rightarrow \N$.
\newline
We let $\alpha^\ast(f,g)$ be defined if there is a minimal $n$ such that $$\sum_{i \leq n}f(i) > 1$$  and then $\alpha^\ast(f,g)$ is the following probability distribution on $\N$:
$$\alpha^\ast(f,g)(m) = \sum\{f(i) \mid {i < n \wedge g(i) = m}\} +\lambda_m$$ where $\lambda_m = 1 - \sum_{i < n}f(i)$ if $g(n) = m$ and $\lambda_m = 0$ if $g(m) \neq m$.
\end{defi}
\begin{obs}{\em $\alpha^\ast$ is externally computable.}\end{obs}
The argument is similar to the proof of the same result for $\bar \alpha$.

\noindent We may now follow the proof of Theorem \ref{3.16} and replace the use of $\bar \alpha$ with a use of $\alpha^\ast$. We then obtain, uniformly in $k\in \N$ and $x \in X$, a fast converging sequence $\mu_{k,x}$ of probability distributions with $f(y)$ as the limit. Thus, adding $\alpha^\ast$ and the limit operator for fast converging probability distributions to our toolbox of computable functions and operators, is sufficient to define any computable function between computable metric spaces.
\begin{rem} We consider $\alpha^\ast$ to be too ad hoc to be of any independent interest. It would be interesting to see if there are alternative operators on metric spaces that are more natural from the point of view of analysis, and that still generate all externally computable total functions.\end{rem}

\subsection*{Acknowledgement} I am grateful to E. M. Alfsen for showing me the construction in \cite{AE}, and to C. Rosendal for making me aware of the results in Kaufman \cite{Kauf}.


\begin{thebibliography}{99}

\bibitem{AE}E. M. Alfsen and Edward G. Effros, {\em Structure in real Banachspaces Part I and II}, Annals of Mathematics Vol 96 No. 1, pp. 98-173 (1972).

\bibitem{BSS} I. Battenfeld, M. Schr\"oder and A. Simpson, {\em A Convenient Category of Domain}, in L. Cardelli, M. Fiore and G. Winskel (eds.) Computation, Meaning and Logic, Articles dedicated to Gordon Plotkin, Electronic Notes in Computer Science 34 (2007).

\bibitem{Kauf} R. Kaufman, {\em Topics on analytic sets}, Fundamenta Mathematicae 139, pp. 215 - 229 (1991).

\bibitem{Kechris} A.S. Kechris, {\em Classical Descriptive Set Theory}, Springer-Verlag (1995).

\bibitem{Ku} K. Kuratowski, {\em Topologie}, vol. 1, Warsawa 1952.

\bibitem{DL}D. Le\v{s}nik, {\em Constructive Urysohn Universal Metric  Space}, Journal of Universal Computer Science, vol 15, no. 6, pp. 1236 - 1263 (2009).

\bibitem{D1} D. Normann, {\em Internal Density Theorems for Hierarchies of Continuous Functionals} , in A. Beckmann, C. Dimitracopoulos and B. L\"owe (eds.) Logic and Theory of Algorithms, Springer LNCS 5028 pp. 467-475 (2008).

\bibitem{D2}D. Normann, {\em Experiments on an Internal Approach to Typed Algorithms in Analysis}, to appear in World Scientific Review Volume edited by S.B. Cooper and A. Sorbi.

\bibitem{D3}D. Normann, {\em A rich hierarchy of functionals of finite types}, Logical Methods in Computer Science, Volume 5(3) 2009 (21 pages).

\bibitem{U1} P. Urysohn, {\em Sur un espace m\'etrique universel}, C.R. Acad. Sci. Paris 180, pp. 803 - 806 (1925).

\bibitem{U2} P. Urysohn, {\em Sur un espace m\'etrique universel},(Edited by P. Alexandroff), Bull. Sci. Math 51, pp. 43 - 64 and 74 - 90 (1927).

\bibitem{Rogers} H. Rogers Jr., {\em Theory of recursive functions and effective computability}, McGraw-Hill (1967).

\bibitem{Sacks} G.E. Sacks, {\em Higher Recursion Theory}, Springer-verlag (1990).

\bibitem{We} K. Weihrauch, {\em  Computational Analysis}, Texts in Theoretical Computer Science, Springer Verlag (2000).

\end{thebibliography}
\end{document}